\documentclass[12pt,a4paper]{amsart}
\usepackage{url,color}
\usepackage{amsmath,amsthm,amsfonts,amssymb,latexsym} 
\newtheorem{theorem}{Theorem}[section]
\newtheorem{proposition}[theorem]{Proposition}
\newtheorem{lemma}[theorem]{Lemma}
\newtheorem{claim}[theorem]{Claim}

\newtheorem{observation}[theorem]{Observation}
\theoremstyle{definition}
\newtheorem{definition}[theorem]{Definition}

\newcommand{\U}{\mathcal U}
\newcommand{\w}{\omega}

\newcommand{\IP}{\mathbb P}

\newcommand{\F}{\mathcal{F}}

\newcommand{\V}{\mathcal{V}}

\newcommand{\bigvid}{\hat{\ \ }}

\newcommand{\uhr}{\upharpoonright}

\newcommand{\name}[1]{\dot{#1}}
\newcommand{\la}{\langle}
\newcommand{\ra}{\rangle}

\newcommand{\forces}{\Vdash}

\newcommand{\hot}{\mathfrak}

\newcommand{\nothing}[1]{}

\title[Productively Menger versus productively Hurewicz]{On the
interplay between productively Menger and productively Hurewicz spaces in models of $\hot b=\hot d$}

\author{Du\v{s}an D. Repov\v{s}, Lyubomyr Zdomskyy}

\address{Faculty of Education and Faculty of Mathematics and Physics, University of Ljubljana, \& Institute of Mathematics, Physics and Mechanics,
1000 Ljubljana, Slovenia.}
\email{dusan.repovs@guest.arnes.si}
\urladdr{http://repovs.splet.arnes.si/}

\address{Institut f\"ur Diskrete Mathematik und Geometrie,
Technische Universit\"at Wien,
Wiedner Hauptstrasse 8—10/104, 1040 Vienna, Austria}
\email{lzdomsky@gmail.com}
\urladdr{http://www.dmg.tuwien.ac.at/zdomskyy/}

\thanks{The first author 
acknowledges the support from the Slovenian Research and Innovation Agency grants 
P1-0292, J1-4031, J1-4001, N1-0278, N1-0114, and N1-0083.
The second  author would
like to thank  the Austrian Science Fund FWF (Grant I 5930)
 for generous support of this research.}

\begin{document}
\maketitle

\begin{abstract}
This article is devoted to the interplay between productively Menger and productively Hurewicz subspaces of the Cantor space.
In particular, we show that 
in the Laver model
 for the consistency of the Borel's conjecture 
 these two notions coincide and characterize Hurewicz spaces.
 On the other hand, it is consistent with  CH that there are productively Hurewicz
 subspaces of the Cantor space which are not productively Menger.
\end{abstract}

\section{Introduction}
This work may be thought of as a continuation of our earlier paper
\cite{RepZdo17}, so we keep our introduction here short and refer the 
reader to that of \cite{RepZdo17}. Except for Theorem~\ref{main_laver}
and its proof at the end of Section~\ref{sec_laver}, we consider 
only zero-dimensional metrizable separable spaces, i.e., subspaces of
the Cantor space $2^\w$ up to a homeomorphism. 

A topological space
$X$ has the  \emph{Menger} property (or, alternatively, is a Menger space)
 if for every sequence $\la \U_n : n\in\omega\ra$
of open covers of $X$ there exists a sequence $\la \V_n : n\in\omega \ra$ such that
each $\V_n$ is a finite subfamily of $\U_n$ and the collection $\{\cup \V_n:n\in\omega\}$
is a cover of $X$. 
We get an equivalent property if we demand that each
$x\in X$ is covered by $\cup\V_n$ for infinitely many $n\in\w$ 
(for this it suffices to split
$\la \U_n : n\in\omega\ra$ into infinitely many mutually disjoint subsequences and
apply the Menger property to each of them).
If in the definition above we additionally require that $\{\cup\V_n:n\in\w\}$
is a \emph{$\gamma$-cover} of $X$
(this means that the set $\{n\in\w:x\not\in\cup\V_n\}$ is finite for each $x\in X$),
then we obtain the definition of the \emph{Hurewicz  property}  introduced
in \cite{Hur27}.  
However, we shall actually use the following characterizations
of these  properties established in \cite{Hur27} (see also
\cite[Theorems~4.3 and 4.4]{COC2}):
$X\subset 2^\w$ is Menger (resp. Hurewicz) if and only if for every continuous 
$f:X\to\w^\w$, the range $f[X]$ is non-dominating (resp. bounded) with respect to the eventual dominance relation $\leq^*$.

One of the basic questions about a topological property is whether
it is preserved by finite products, which led to the following 
definitions introduced in \cite{MilTsaZdo14}: A topological space
$X$ is \emph{productively Hurewicz} (resp. \emph{productively Menger}),
if $X\times Y$ is Hurewicz (resp. Menger) for all Hurewicz
(resp. Menger) spaces $Y$. Since singletons are Hurewicz, productively 
Hurewicz (resp. productively Menger) spaces are Hurewicz (resp. Menger).
If $\hot b<\hot g$, there are productively Menger spaces which are not even Hurewicz, see the discussion at the beginning of
\cite[p.~10]{SzeTsa17}. However, if $\hot b=\hot d$, then
productively Menger spaces are productively Hurewicz by \cite[Theorem~4.8]{SzeTsa17}. 

In this paper 
we show that the statement 
``$\hot b=\hot d$ and classes of productively Hurewicz and productively Menger spaces coincide'' is independent from ZFC. More precisely, in one direction  we use the key  lemma of \cite{Lav76} in the style of \cite{MilTsa10} and prove the following

\begin{theorem} \label{main_laver}
In the Laver model for the consistency of the Borel's conjecture,
  Hurewicz spaces are productively Menger (and hence also productively Hurewicz) in the realm of 
subspaces of $2^\w$. 

Consequently, the product $X\times Y$ of a Hurewicz space $X$ and a
Menger space $Y$ is Menger if it is Lindel\"of.
\end{theorem}
The theorem above is an improvement of \cite[Theorem~1.1]{Zdo20}, as
follows from \cite[Theorem~4.8]{SzeTsa17}.

The next result, whose proof relies on ideas from \cite{SzeTsa17},
shows that the conclusion of  Theorem~\ref{main_laver} is not 
a consequence of CH, so also not of $\hot b=\hot d$, the latter being 
the assumption in \cite[Theorem~4.8]{SzeTsa17}.

\begin{theorem}\label{main_ch}
The existence of a productively Hurewicz space which is not productively Menger is consistent with CH.
\end{theorem} 

We do not know whether the conclusion of Theorem~\ref{main_ch}
is actually a consequence  of CH. The space we construct in 
the proof of Theorem~\ref{main_ch} answers \cite[Problems~7.6,7.8]{SzeTsa17} in the negative. 

We refer the reader to \cite{Bla10} for the definitions of  cardinal
characteristics we use, 
\cite{Eng} for topological notions we use but not define,
and to 
\cite{RepZdo17} for more motivation behind 
the research done in this paper.

\section{Productively Hurewicz spaces in the Laver model} \label{sec_laver}

This section is mainly devoted to the proof of Theorem~\ref{main_laver}.

\begin{definition} \label{def01}
$X\subset 2^\w$ is said to satisfy  property
$(\dagger)$, if for every function $M$ assigning to each countable
subset $Q$ of $X$ a Menger subset $M(Q)\cap Q=\emptyset$ of $ 2^\w$,
there exists a family $\mathsf Q\subset [X]^\w$ of size $|\mathsf
Q|=\w_1$  such
that
 $X\subset\bigcup_{Q\in\mathsf Q}(2^\w\setminus M(Q))$. 
 \hfill $\Box$
\end{definition}

Let us note that under CH any $X\subset 2^\w$ satisfies $(\dagger)$.

The following lemma is the key part of the proof of
Theorem~\ref{main_laver}. Its proof is reminiscent of that of
\cite[Theorem~3.2]{MilTsa10}.
 We will use the notation from \cite{Lav76} with  only differences being  that
smaller conditions in a forcing poset  are stronger, i.e., 
 carry more information about the generic filter, and the ground model is (nowadays standardly) denoted by $V$.
We shall work in $V[G_{\w_2}]$, where $V$ 
satisfies GCH,
  $G_{\w_2}$ is
$\IP_{\w_2}$-generic and $\IP_{\w_2}$ is the iteration of length
$\w_2$ with countable supports of the Laver forcing, see
\cite{Lav76} for details. For $\alpha\leq \w_2$ 
 we shall denote $G_{\w_2}\cap
\IP_\alpha$ by $G_\alpha$, For a Laver tree $T\subset\w^{<\w}$ we denote by $T\la
0\ra$ its root. If $s\in T$, $s\geq T\la 0\ra$, then 
we denote by $S_T(s)$ the family of all immediate successors of
$s$ in $T$.

As usually, 
$\forall^*$ means ``for all but finitely many''.

A subset $C$ of $\w_2$ is called an \emph{$\w_1$-club} if it is
unbounded and for every $\alpha\in\w_2$ of cofinality $\w_1$, if
$C\cap\alpha$ is cofinal in $\alpha$ then $\alpha\in C$.

\begin{lemma} \label{laver}
In the Laver model every  $X\subset 2^\w$ with the Hurewicz property satisfies $(\dagger)$. 
\end{lemma}
\begin{proof}
First let us work in $V[G_{\w_2}]$.
Let $M$ be such as in the definition of $(\dagger)$.
By a standard closing-off argument
there exists an $\w_1$-club $C\subset \w_2$ such that for every $\alpha\in C$
the following conditions are satisfied:
\begin{itemize}
\item $X\cap V[G_\alpha]\in V[G_\alpha]$\footnote{Note that for a set $A\in V[G_{\w_2}],$ the inclusion $A\subset V[G_\alpha]$ does not imply $A\in V[G_\alpha]$. For example,
if $A\subset\w$, then $A\subset V$ because $\w\subset V$. Thus, in $V[G_{\w_2}]$
there are $\w_2$-many subsets of $\w$, each of them is a subset of $V$, but only $\w_1$-many of them are \emph{elements} of $V$.};
\item For every $Q\in [X]^\w\cap V[G_\alpha]$ 
and every continuous map $\phi:M(Q)\to\w^\w$ coded in $V[G_\alpha]$,
there exists $h\in\w^\w\cap V[G_\alpha]$ such that 
$h\not\leq^*\phi(y)$ for any $y\in M(Q)$; and
\item For every $E\in [ 2^\w]^\w\cap V[G_\alpha] $ disjoint from $X$
there exists a $G_\delta$ set $O\supset E$ coded in $V[G_\alpha]$ such that
$O\cap X=\emptyset$.
\end{itemize}
The existence of a $G_\delta$-set required in the last item above 
is a well-known consequence of the Hurewicz property, see, e.g.,
\cite[Theorem~5.7]{COC2}.

Let us fix $\alpha\in C$.
We claim that
$X\subset \bigcup_{Q\in \mathsf Q} (2^\w\setminus M(Q))$,
where $\mathsf Q=[X]^\w\cap V[G_\alpha],$ which
 would complete our proof. By \cite[Lemma~11]{Lav76},
there is no loss of generality
in assuming that $\alpha=0$: 
We still have GCH in $V[G_\alpha]$,
and the quotient forcing 
$\IP_{[\alpha,\w_2)}$ is again the iteration of length
$\w_2$ with countable supports of the Laver forcing, defined in $V[G_\alpha]$. With this convention we have $V[G_\alpha]=V$, and hence the three items considered above
hold for $V$ instead of $V[G_\alpha]$.
 
Now we start working in $V$. Let $\name{X}$ and $\name{M}$ be $\IP_{\w_2}$-names for $X$ and $M$, respectively.
Suppose that, contrary to our claim, there exists $p\in G_{\w_2}$ and a $\IP_{\w_2}$-name $\name{x}$
such that $p$ forces $  \name{X}$ to be Hurewicz and $\name{x}\in\name{X}\setminus \bigcup_{Q\in \mathsf Q}(2^\w\setminus \name{M}(Q))$.
 Applying \cite[Lemma~14]{Lav76}
to the sequence $\la \name{a}_i:i\in\w\ra$ such that $\name{a}_i=\name{x}$ for all $i\in\w$,
we get 
a condition $p'\leq p$ such that $p'(0)\leq_0 p(0)$, and a finite set $U_s$
of reals for every $s\in p'(0)$ with $p'(0)\la 0\ra\leq s$, such that for each
  $s$ as above the following property is satisfied:
\smallskip

 \begin{itemize}
\item[$(1)_{p',s}$] 
     $
  \forall n\in\w\: \forall^* t\in S_{p'(0)}(s)\: \big( p'(0)_t\bigvid p'\uhr[1,\w_2)\forces \exists u\in U_s\: (\name{x}\uhr n =u\uhr n)\big)$. 
\end{itemize}
Repeating the argument from the proof of \cite[Lemma~2.3]{Zdo20}, namely the part before equation $(5)$ there, we could, passing to a 
stronger condition, if necessary, assume that 
$U_s\subset X$ for all $s\in p'(0)$ such that $p'(0)\la 0\ra\leq s$.
For this the third assumption on ordinals $\alpha\in C$ is crucial.

\begin{claim}\label{cl01}
If $s\in p'(0)$, $p'(0)\la 0\ra\leq s$, $u\in U_s$, and $n_0\in\w$ are such that 
there is no $r\leq p'(0)_s\bigvid p'\uhr[1,\w_2)$
forcing $\name{x}\uhr n_0= u\uhr n_0$, then 
$(1)_{p',s}$ is still satisfied when $U_s$ is replaced with 
$U_s\setminus\{u\}$.
\end{claim}
\begin{proof}
Suppose that  $(1)_{p',s}$ with $U_s\setminus\{u\}$ instead of $U_s$ 
fails for for some  $n\in\w$. There is no loss of 
generality in assuming that $n\geq n_0$. This means that 
the set 
$$T=\big\{t\in S_{p'(0)}(s)\: : \: p'(0)_t\bigvid p'\uhr[1,\w_2)\not\forces \exists v\in U_s\setminus\{u\}\: (\name{x}\uhr n =v\uhr n)\big\}$$
is infinite. For each $t\in T$ find
$r_t\leq p'(0)_t\bigvid p'\uhr[1,\w_2)$ such that 
$$r_t\forces \forall v\in U_s\setminus\{u\}\: (\name{x}\uhr n \neq v\uhr n).$$
Let $t_1\in T$ be such that 
$$p'(0)_{t_1}\bigvid p'\uhr[1,\w_2)\forces \exists v\in U_s\: (\name{x}\uhr n =v\uhr n)\big).$$
From the two formulas displayed above it follows that 
$r_{t_1}\forces \name{x}\uhr n = u\uhr n,$ and hence also 
$r_{t_1}\forces \name{x}\uhr n_0 = u\uhr n_0$ because $n\geq n_0$,
which is impossible by our assumption. 
\end{proof}

For a subset $K$ of $2^\w$ we denote by 
 $O_n(K)$ the set $\{z\in 2^\w: z\uhr n=y\uhr n $ for some $y\in K\} $. 

\begin{claim}\label{cl02}
For every $s\in p'(0)$, $p'(0)\la 0\ra\leq s$, and  
 every
   $t\in S_{p'(0)}(s)$ there exists 
$U'_t\subset U_t$ such that 
\begin{itemize}
\item $(1)_{p',t}$ is still satisfied when $U_t$ is replaced with 
$U'_t$; and
\item for every $n\in\w$  and all but finitely many 
$t\in S_{p'(0)}(s)$ we have $U'_t\subset O_n(U_s)$.
\end{itemize}
\end{claim}
\begin{proof}
Fix  $n\in \w$.
Then by $(1)_{p',s}$ there exists $A\in [S_{p'(0)}(s)]^{<\w}$ such that 
\begin{equation}\label{eq02}
 p'(0)_t\bigvid p'\uhr[1,\w_2)\forces\exists u\in U_s (u\uhr n=\name{x}\uhr n)    
\end{equation}
 for every $t\in S_{p'(0)}(s)\setminus A$.  Note that (\ref{eq02}) implies that if $t\in S_{p'(0)}(s)\setminus A$, $w\not\in O_n(U_s)$, and 
$r\leq p'(0)_t\bigvid p'\uhr[1,\w_2)$, then $r$ forces $w\uhr n\neq\name{x}\uhr n$. Applying Claim~\ref{cl01} we conclude that $(1)_{p',t}$ is satisfied for
$t\in S_{p'(0)}(s)\setminus A$ if we replace $U_t$ with $U_t\cap O_n(U_s)$.

Applying the same argument recursively for every $n\in\w$ we can get an increasing sequence $\la A_n:n\in\w\ra$ of finite subsets of $S_{p'(0)}(s)$ with $\bigcup_{n\in\w}A_n=S_{p'(0)}(s)$ such that $(1)_{p',t}$ is satisfied for
$t\in S_{p'(0)}(s)\setminus A_n$ if we replace $U_t$ with $U_t\cap O_n(U_s)$.
It remains to set $U'_t=U_t$  for all $t\in A_0$ and $U'_t=U_t\cap O_n(U_s)$  for all $t\in A_{n+1}\setminus A_n$ and note that these $U'_t$ are as required.
\end{proof}

 \begin{claim}\label{cl03}
 Let $K\subset 2^\w$ be compact, and for every $i\in\w$ let 
 $\la U^i_m:m\in\w\ra$ be a sequence of finite subsets of $2^\w$ such that 
 $$ \forall i\in\w\: \forall n\in\w\: \forall^* m\in\w  \: (U^i_m\subset O_n(K)) .$$
 Then for every $i\in\w$ there exists $m_i\in\w$ such that
 $K\cup\bigcup_{i\in\w}\bigcup_{m\geq m_i}U^i_m$ is compact. 
 \end{claim}
\begin{proof}
It suffices to choose $m_i$ such that $U^i_m\subset O_i(K)$
for all $m\geq m_i$, the standard details are left to the reader.
\end{proof}
 
After three auxiliary claims above, we are in a position to proceed with the proof of Lemma~\ref{laver}. Combining Claims~\ref{cl02} and \ref{cl03}, we can get a Laver condition $T\leq_0 p'(0)$ 
and $U'_t\subset U_t$ for every splitting node 
$t\in T$ such that
\begin{itemize}
\item[$(i)$] letting $p''=T\bigvid p'\uhr[1,\w_2)$, we have $p''\in G_{\w_2}$ and 
$(1)_{p'',t}$ is satisfied for all  splitting nodes 
$t\in T=p''(0)$; 
\item[$(ii)$] For every $s\geq T\la 0 \ra$, $n\in\w$, and 
all but finitely many $t\in S_T(s)$ we have $U'_t\subset O_n(U'_s)$; and
 \item[$(iii)$]
$K_m:=\bigcup\big\{U'_t\: :\:t\in T,T\la 0\ra\leq t,|t|\leq m\big\}$
is a compact subset of $2^\w$ for all   $m\in\w$.
(Note that
$K_m=\emptyset$ for $m<m_0:=|T\la 0\ra|=|p'(0)\la 0\ra|$.)
\end{itemize}
Set $Q_*=\bigcup_{m\in\w}K_m \in [X]^\w$ and consider the map 
$\phi:2^\w\setminus Q_*\to\w^\w$ defined as follows:
$$\phi(z)(m)=\min\{n\in\w:z\not\in O_n(K_m)\}.$$
Since $K_m$ is closed and $z\not\in K_m$, $\phi$ is well-defined. 
The second item 
describing properties of ordinals in $C$ yields 
$h\in\w^\w\cap V$ such that 
$h\not\leq^* \phi(y)$ for any $y\in \name{M}(Q_*)$. It follows that
$$p''\forces\name{x}\in \name{X}\setminus \bigcup_{Q\in\mathsf Q}(2^\w\setminus \name{M}(Q))=\name{X}\cap\bigcap_{Q\in\mathsf Q}\name{M}(Q)\subset \name{X}\cap \name{M}(Q_*),$$
and hence 
\begin{equation}\label{contra}
p''\forces h\not\leq^*\phi(\name{x}).
\end{equation}
On the other hand, recursively removing finitely many immediate successors of every splitting node $s$ of $T=p''(0)$, using $(i)$ and $(ii)$, we can  get a Laver tree $T'\leq_0 T$ such  that
\begin{equation}\label{contra1}
T_t\bigvid p'\uhr [1,\w_2)\forces\exists u\in U'_s\:(\name{x}\uhr h(m)=u\uhr h(m))
\end{equation}
for all $m\geq m_0$, $s\in T'\cap\w^m$, 
and $t\in S_{T'}(s)$. Equation~(\ref{contra1})
gives
$T'_t\bigvid p'\uhr [1,\w_2)\forces \name{x}\in O_{h(m)}(K_m) $
for all $m,s,t$ as above, and therefore 
$$T' \bigvid p'\uhr [1,\w_2)\forces\forall m\geq m_0\:(\name{x}\in O_{h(m)}(K_m)),$$
or, equivalently,
$$T' \bigvid p'\uhr [1,\w_2)\forces\forall m\geq m_0\:(\phi(\name{x})>h(m)),$$
which together with $T'\bigvid p'\uhr[1,\w_2)\leq p''$
contradicts (\ref{contra}) and thus finishes our proof.
\end{proof}

The next lemma demonstrates the relation between $(\dagger)$
and products with Menger spaces.

\begin{lemma} \label{main1}
 Suppose that $\hot b>\w_1$. Let $Y \subset  2^\w$ be a Menger  space
and $X\subset  2^\w$ a Hurewicz space  satisfying $(\dagger)$.  Then $X\times Y$ is Menger.
\end{lemma}
\begin{proof}
Fix a sequence $\la \U_n:n\in\w\ra$ of countable covers of $X\times Y$
by clopen subsets of $ 2^\w\times Y$.
For every $Q\in [X]^\w $ using that $Q\times Y$ is Menger,
we can find  a sequence $\la \V^Q_n:n\in\w\ra$ such that
$\V^Q_n\in [\U_n]^{<\w}$ for all $n\in\w$ and 
$Q\times Y\subset W^Q:=\bigcap_{m\in\w}\bigcup_{n\geq m}\cup\V^Q_n$.
Then $(2^\w\times Y)\setminus W^Q$ is Menger being an $F_\sigma$-subset of the Menger space $2^\w\times Y$, and it is
disjoint from $Q\times Y$, and hence its projection
$M(Q)$ onto the first coordinate is a Menger subspace of 
$2^\w$ disjoint from $Q$.  

Since $X$ satisfies $(\dagger)$, there exists
$\mathsf Q\subset [X]^\w$ of size $|\mathsf Q|=\w_1$
such that $X\subset\bigcup_{Q\in\mathsf Q}(2^\w\setminus M(Q))$.
Since $|\mathsf Q|<\hot b$, there exists a 
sequence 
$\la \V_n:n\in\w\ra$ such that
$\V_n\in [\U_n]^{<\w}$ for all $n\in\w$, and for every
$Q\in\mathsf Q$ there exists $m(Q)\in\w$ such that 
$\V^Q_n\subset\V_n$ for all $n\geq m(Q)$.
We claim that $X\times Y\subset\bigcup_{n\in\w}\cup\V_n$.
Indeed, given $\la x,y\ra\in X\times Y$, find $Q\in\mathsf Q$
such that $x\not\in M(Q)$. This implies $\la x,y\ra \in W^Q$, and hence
there exists $n\geq m(Q)$ with 
$\la x,y\ra\in\cup\V^Q_n$, consequently $\la x,y\ra\in\cup\V_n$
because $\V^Q_n\subset\V_n$, which completes our proof.
\end{proof}

Finally, we can prove the characterization of Hurewicz subspaces of $ 2^\w$ which holds in the
Laver model and implies Theorem~\ref{main_laver}.

\begin{proposition}\label{prop_char}
In the Laver model, for a subspace $X$ of $ 2^\w$ the following conditions are equivalent:
\begin{enumerate}
\item $X$ is Hurewicz; 
\item $X$ satisfies $(\dagger)$; 
\item $X$ is productively Menger; and 
\item $X$ is productively Hurewicz. 
\end{enumerate}
 \end{proposition}
\begin{proof}
The implication $(1)\to(2)$  is established in Lemma~\ref{laver}.
 The implication $(2)\to (3)$ is proved in Lemma~\ref{main1}
and thus requires only $\hot b>\w_1$. And finally, $(3)\to (4)$ 
follows from \cite[Theorem~4.8(2)]{SzeTsa17} because $\hot b=\hot d$ holds
in the Laver model, while $(4)\to (1)$ is obvious.
\end{proof}

Finally, 
by nearly the same argument as at the end of \cite{RepZdo17} we can prove that  Theorem~\ref{main_laver}
follows from Proposition~\ref{prop_char}. 
Again, we present
its proof for the sake of completeness.
A family $\F\subset[\w]^\w$ is called a \emph{semifilter} if for every
$F\in\F$ and $X\subset \w$, if $|F\setminus X|<\w$
then $X\in\F$. 
Each semifilter is considered with the topology inherited 
from the Cantor space $2^\w$ which we identify with $\mathcal P(\w)$ via characteristic functions.

The proof of  the second part of Theorem~\ref{main_laver}  uses   characterizations
of the Hurewicz and Menger properties  obtained in \cite{Zdo05}.
Let  $u=\la U_n : n\in\omega\ra$ be a sequence of subsets of a set $X$.
For every $x\in X$ let  $I_s(x,u,X)=\{n\in\omega:x\in U_n\}$. If every
$I_s(x,u,X)$ is infinite (the collection of all such sequences $u$ will be denoted
by $\Lambda_s(X)$), then we shall denote by $\mathcal U_s(u,X)$
the smallest semifilter on $\omega$ containing all $I_s(x,u,X)$.
By \cite[Theorem~3]{Zdo05},  a Lindel\"of topological space $X$  is Hurewicz (Menger) if and only if
for every  $u \in\Lambda_s(X)$ consisting of open sets,
  the semifilter $\mathcal U_s(u,X)$ is Hurewicz (Menger).
The proof given there also works if we consider only those
$\la U_n : n\in\omega\ra\in\Lambda_s(X)$ such that all $U_n$'s belong to a given base of
$X$.

\medskip

\noindent\textit{Proof of Theorem~\ref{main_laver}.} \
Suppose that $X$ is Hurewicz, $Y$ is Menger,   $X\times Y$ is Lindel\"of,
and fix  $w=\la U_n\times V_n :n\in\w\ra\in\Lambda_s(X\times Y)$
consisting of open sets.
Set  $u=\la U_n:n\in\w\ra$,  $v=\la V_n:n\in\w\ra$, and note that
$u\in\Lambda_s(X)$ and $v\in\Lambda_s(Y)$.
It is easy to see
that
$$\U_s(w,X\times Y)=\{A\cap B: A\in \U_s(u,X), B\in \U_s(v,Y)\},$$
and hence $\U_s(w,X\times Y)$ is a continuous image of
$\U_s(u,X)\times \U_s(v,Y)$. By \cite[Theorem~3]{Zdo05} 
$\U_s(u,X)$ and $\U_s(v,Y)$ are Hurewicz and Menger, respectively,  considered as subspaces of $2^\w$, and hence their product is a Menger space by
Proposition~\ref{prop_char}. Thus $\U_s(w,X\times
Y)$ is Menger, being a continuous image of a Menger space. It now
suffices to use \cite[Theorem~3]{Zdo05} again, in the other direction. \hfill $\Box$
\medskip

\section{Productively Hurewicz spaces in models of CH}\label{sec_ch}

In this section we prove Theorem~\ref{main_ch}.

Suppose that CH holds in the ground model $V$ and 
fix $Y=\{ y_\alpha:\alpha<\w_1\}\subset [\w]^\w$ such that
\begin{itemize}
\item
 $y_\beta\subset^*y_\alpha$ for all $\beta>\alpha$;
 \item $y_{\alpha+1}\subset y_\alpha$ for all $\alpha$;
 \item $y_\alpha\setminus y_{\alpha+1}$ is infinite for all $\alpha$; and
\item For every $y\in[\w]^\w$ there exists $\alpha$ with $y_\alpha\not\leq^* y$, where each element of $[\w]^\w$ is identified with its 
 increasing enumerating function. 
 \end{itemize}
Thus $Y$ is an unbounded tower in the terminology of \cite{OreTsa11}.

 In what follows, we shall work in $V[G]$, where $G$ is 
 $\mathit{Fn}(\w_1,2)$-generic over $V$. 
 Here $\mathit{Fn}(\w_1,2)$
 is the standard poset adding $\w_1$ many Cohen reals over $V$.
 
 It is well-known that $Y=\{y_\alpha:\alpha<\w_1\}$ is Menger 
 in $V[G]$, see, e.g., \cite[Theorem~11]{SchTal10}.
 Moreover, $Y$ is also unbounded in $V[G]$ since 
 Cohen reals preserve the unboundedness of ground model unbounded sets.
 Fix an enumeration 
 $[\w]^\w=\{z_\alpha:\alpha\in\w_1\}$ and for every
 $\alpha$ pick  $x_\alpha \in [y_\alpha]^\w$
 such that $x_\alpha\supset y_{\alpha+1},$ $|y_\alpha\setminus x_\alpha|=|x_\alpha\setminus y_{\alpha+1}|=\w$, and
 $z_\alpha\leq^* (y_\alpha\setminus x_\alpha)$.

Since 
$x_\beta\subset y_\beta\subset^* y_{\alpha+1}\subset x_\alpha$
for any $\alpha<\beta$, we conclude that 
  $\{x_\alpha:\alpha<\w_1\}$ is an unbounded tower as well, and hence $X:=\{x_\alpha:\alpha<\w_1\}\cup [\w]^{<\w}$ is productively 
 Hurewicz\footnote{ $X$ is also a $\gamma$-set by the main result of 
 \cite{OreTsa11}.} 
 by \cite[Theorem~6.5(1)]{MilTsaZdo14}.   Now, Theorem~\ref{main_ch} is a direct consequence of the following

 \begin{observation}\label{obv}
 $X\times Y$ is not Menger, and therefore $X$ is not productively Menger.
 \end{observation}
\begin{proof}
Let $\oplus$ be the coordinate-wise addition modulo $2$ in $2^\w$, i.e., the standard
operation turning $2^\w$ into a topological group.
We shall show that $X\oplus Y$ is not a Menger subspace of $2^\w$.
Since no dominating subset of $[\w]^\w$
is Menger (see, e.g., \cite[Theorem~4.4]{COC2}) and $z_\alpha\leq^*
y_\alpha\setminus x_\alpha=y_\alpha\oplus x_\alpha$, it remains to 
check that $X\oplus Y\subset [\w]^\w$. This is done through a routine 
consideration of $x_\alpha\oplus y_\beta$ for all possible $\alpha,\beta$ below.

1.  $\alpha<\beta$. In this case $x_\alpha\setminus y_{\alpha+1}\subset^* x_\alpha\setminus y_\beta\subset x_\alpha\oplus y_\beta$,
and $x_\alpha\setminus y_{\alpha+1}$ is infinite by the choice of 
$x_\alpha.$

2. $\alpha=\beta$. Then $y_\alpha\setminus x_\alpha=x_\alpha\oplus y_\alpha$ is infinite by the choice of $x_\alpha$.

3. $\alpha>\beta$. In this case 
$x_\alpha\subset y_\alpha\subset^* y_\beta$ and $y_\alpha\setminus x_\alpha$ is infinite. Thus, $y_\beta\setminus x_\alpha$ is infinite as well, and the latter difference is included into $x_\alpha\oplus y_\beta$.
\end{proof}
\medskip 
 
\noindent \textbf{Acknowledgments.} The second author is 
grateful to members of the Ko\v{s}ice topology and logic group for
the possibility to present the results of this work in their seminar in Spring 2021 and many valuable comments during that presentation. 
We thank the referee for comments and suggestions.

\end{document}